\documentclass[12pt]{amsart}
\usepackage{amsmath,amsfonts, mathtools,amssymb,amscd,latexsym,amssymb,amsthm,adjustbox,mathrsfs,amsbsy,bm,tikz-cd,indentfirst,makecell,graphicx, verbatim,calc,enumerate,xy}
\usepackage[top=1in,bottom=1.3in,left=1.1in,right=1.1in]{geometry}

\theoremstyle{plain}
\newtheorem{theorem}{Theorem}[section]
\newtheorem{lemma}[theorem]{Lemma}

\newtheorem{corollary}[theorem]{Corollary}

\newtheorem{remark}[theorem]{Remark}

\newtheorem{example}[theorem]{Example}
\newtheorem{point}[theorem]{}

\newcommand{\F}{\mathbb{F}}
\newcommand{\m}{\mathfrak{m}}
\newcommand{\n}{\mathfrak{n}}

\newcommand{\R}{\mathcal{R}}

\newcommand{\N}{\mathbb{N}}

\newcommand{\wt}{\widetilde }

\newcommand{\depth}{\operatorname{depth}}

\newcommand{\ord}{\operatorname{ord}}

\newcommand{\projdim}{\operatorname{projdim}}

\newcommand{\Ext}{\operatorname{Ext}}
\newcommand{\Tor}{\operatorname{Tor}}

\usepackage{graphicx}
\usepackage{subfigure}
\usepackage{epsfig}
\usepackage{epstopdf}
\usepackage{float}
\theoremstyle{plain}

\usepackage{xcolor}
\usepackage{titlesec}

\titleformat{name=\section}{}{\thetitle.}{0.5em}{\centering\scshape\large}

\begin{document}
\title{\large \textbf{ Quasi-pure resolutions and some lower bounds of Hilbert coefficients of Cohen-Macaulay modules. }}
\author{Tony J. Puthenpurakal}
\email{tputhen@math.iitb.ac.in}

\author{Samarendra Sahoo}
\email{204093008@math.iitb.ac.in}

\address{Department of Mathematics, IIT Bombay, Powai, Mumbai 400 076, India}
\date{\today}

\subjclass{Primary 13A30, 13C14; Secondary 13D40, 13D07}
\keywords{Associated graded rings and modules,  strict complete intersections, Gorenstein rings, graded resolutions}

\begin{abstract}
    Let $(A,\mathfrak{m})$ be a Gorenstein local ring and let $M$ be a finitely generated Cohen Macaulay $A$ module. Let $G(A)=\bigoplus_{n\geq 0}\mathfrak{m}^n/\mathfrak{m}^{n+1}$ be the associated graded ring of $A$ and  $G(M)=\bigoplus_{n\geq 0}\mathfrak{m}^nM/\mathfrak{m}^{n+1}M$ be the associated graded module of $M$. If $A$ is regular and if $G(M)$ has a quasi-pure resolution then we show that $G(M)$ is Cohen-Macaulay.
    If $G(A)$ is Cohen-Macaulay and if $M$ has finite projective dimension then we give lower bounds on $e_0(M)$ and $e_1(M)$. Finally let $A = Q/(f_1, \ldots, f_c)$ be a strict complete intersection with $\ord(f_i) = s$ for all $i$. Let $M$ be an Cohen-Macaulay module with $\operatorname{cx}_A(M)=r< c$. We give lower bounds  on $e_0(M)$ and $e_1(M)$.
     \end{abstract}

\maketitle

\section{Introduction}

 Let $(A,\m)$ be a Cohen Macaulay local ring of dimension $d$ with residue field $k=A/\m$ and let $M$ be finitely generated  $A$-module. Let $G(A)=\bigoplus_{n\geq 0}\m^n/\m^{n+1}$ be the associated graded ring of $A$ and  $G(M)=\bigoplus_{n\geq 0}\m^nM/\m^{n+1}M$ be the associated graded module of $M$ considered as a graded $G(A)$-module.
  Let $\lambda(E)$ denotes length of an $A$-module $E$.
   Let $M$ be an $A$-module of dimension $r$ . The Hilbert series of $M$ is $$H_M(z)=\sum_{i\geq 0}\lambda(\m^iM/\m^{i+1}M)z^i=\frac{h_M(z)}{(1-z)^r}.$$ Where $h_M(z)=h_0+h_1z+\ldots +h_tz^t$ is called $h$-polynomial of $M$. The integer $e_i(M)=h^i_M(1)$ is called as  the $i^{th}$-Hilbert coefficient of $M$, where $h^i_M(z)$ is $i$-th derivative of $h_M(z).$

\textbf{I:} \emph{The case when $A$ is regular local.}

Let $A=K[[x_1,\ldots ,x_n]]$ be a power series ring over the field $K$ and $M$ be a Cohen-Macaulay $A$-module.  In \cite{Pure} the author proved that if $G(M)$ has a pure resolution then $G(M)$ is Cohen Macaulay. When $A$ is not equicharacteristic this result is proved in \cite{ananthnarayan2023associated}.

 Let $R$ be a standard graded polynomial ring over a field and $E$ be a finitely generated graded $R$ module. Define $$\alpha_i=\operatorname{max}\{ n\, |\, \operatorname{Tor}_i(k,E)_n\neq 0\}$$ and $$\gamma_i= \operatorname{min}\{ n\,|\, \operatorname{Tor}_i(k,E)_n\neq 0 \}.$$ We say $E$ has a $\textit{pure resolution}$ if $\alpha_i=\gamma_i$ for all $i\geq 0$ and quasi pure resolution if $\gamma_i\geq \alpha_{i-1}$ for all $i\geq 1.$ For examples of pure and $\textit{quasi-pure resolution}$ see \ref{def:example}.

Our first result is:
\begin{theorem}
\label{Thm:1st}
    Let $(A,\m)$ be a regular local ring and $M$ be a  Cohen Macaulay $A$ module. Set $R=G(A)=k[X_1,\ldots ,X_n]$. Set $p=\operatorname{projdim}(G(M))$. If $G(M)$ is not Cohen Macaulay then $\alpha_p(M)<\alpha_{p-1}(M).$
\end{theorem}

 An  easy application of this theorem is:
 \begin{corollary}
 \label{Thm:2nd}
      Let $(A,\m)$ be a regular local ring and $M$ be a finitely generated Cohen Macaulay $A$ module. If $G(M)$ has quasi-pure resolution then $G(M)$ is Cohen Macaulay.
 \end{corollary}

 \textbf{II:} \emph{Cohen-Macaulay modules with finite projective dimension.}

 If $A$ is not regular then note that even if $\projdim_A M$ is finite it might very well happen that $\projdim_{G(A)} G(M)$ is infinite. For an example if $A$ is Cohen-Macaulay with $\depth G(A) = 0$ then $\projdim_{G(A)} G(M)$ is infinite if $\dim M < \dim A$. However if $A$ is Gorenstein and $G(A)$ is Cohen-Macaulay then if $M$ is a Cohen-Macaulay module with finite projective dimension then we can give bounds of multiplicity $e_0(M)$ and the first Hilbert coefficient $e_1(M)$ of $M$. Let $\operatorname{reg}(G(A))$ denote the regularity of $G(A)$. We prove
 \begin{lemma}
\label{Thm:lowerbound-intro}
Let $(A,\m)$ be a Gorenstein local ring. Assume $G(A)$ is Cohen Macaulay. Let $M$ be an Cohen Macaulay module of dimension $r$ with finite projective dimension. Then $e_0(M)\geq \mu(M)+ c$ and $e_1(M)\geq \binom{c+1}{2},$ where $c=\operatorname{reg}(G(A)).$
\end{lemma}
We also show that if the lower bound for $e_1(M)$ is attained then necessarily $G(M)$ is Cohen-Macaulay, see Theorem \ref{Thm:reg}.

\textbf{III:} \emph{Cohen-Macaulay modules over strict complete intersections.}

Next, we consider the following situation. Let  $(Q,\n)$ be a regular local ring with infinite residue field, $f_1,\ldots ,f_c\in \n^2$ of order $s$ such that $f_1^*,\ldots ,f_c^* $ is a $G(Q)$-regular sequence.  Let $A=Q/(f_1,\ldots ,f_c)$ and let $M$ be an Cohen Macaulay $A$-module with $\operatorname{cx}_A(M)=r< c$.

\begin{theorem}
\label{Thm:invariant-intro}(with hypotheses as above)
Then $e_0(M)\geq \mu(M)+\alpha$ and $e_1(M)\geq \binom{\alpha+1}{2}$,  where $\alpha=(c-r)(s-1).$ Moreover if $e_1(M)=\binom{\alpha+1}{2}$ then $G(M)$ is Cohen Macaulay.
\end{theorem}

We now describe in brief the contents of this paper.
 In section 2 we describe some preliminary results that we need. In section 3 we sketch a proof of Theorems \ref{Thm:1st} and Corollary \ref{Thm:2nd}.  We also give some examples showing that our hypothesis about $M$ is optimal. In section 4 we give a proof of Lemma \ref{Thm:lowerbound-intro}. Finally, in section five we prove Theorem \ref{Thm:invariant-intro}.

\section{Preliminaries}
Throughout this paper all rings considered are Noetherian and all modules considered (unless stated otherwise) are finitely generated.

\begin{point}
      \normalfont
     Let $(A,\m)$ be a local ring and let $M$ be an $A$-module. Define \\ $L(M)=\bigoplus_{n\geq 0}M/\m^{n+1}M$. Let $\R=A[\m t]$ be Rees ring and
  $\mathcal{R}(M)=\bigoplus_{n\geq 0}\m^nMt^n$ be Rees module of $M$. The Rees ring $\mathcal{R}$ is a subring of $A[t]$. So $A[t]$ is an $\mathcal{R}$ module. Therefore $M[t] = M\otimes_A A[t]$ is an $\R$-module. The exact sequence $$0\to \mathcal{R}(M)\to M[t]\to L(M)(-1)\to 0 $$ defines an $\mathcal{R}$ module structure on  $L(M)(-1)$ and hence on $L(M)$(for more details see (\cite{Part1}, definition 4.2)). Note that $L(M)$ is not a finitely generated $\R$-module.
\end{point}
\begin{point}
     \normalfont
   An element $x\in  \m$ is called $M$-superficial with respect to $\m$ if there exists $c\in \N$ such that for all $n \geq c$, $(\m^{n+1}M :_Mx)\cap \m^cM = \m^nM$. If depth $M>0$ then one can show that a $M$-superficial element is $M$-regular. Furthermore $(\m^{n+1}M :_Mx) = \m^nM$ for $n\gg0.$ Superficial elements exist if the residue field is infinite.

    A sequence $x_1, \ldots ,x_r$  in $(A, \m)$ is said to be $M$-superficial sequence if $x_1$ is $M$-superficial and $x_i$ is $M/(x_1,\ldots ,x_{i-1})M$-superficial for $2\leq i \leq r .$
\end{point}
\begin{point}
       \normalfont
      Let $x_1, \ldots ,x_n $ be a sequence in $A$. For an $A$-module $M$, we denote $K\textbf{.}(\underline{x},M)$ as Koszul complex of $\underline{x}$ with respect to $M$ and $H_i(K\textbf{.}(\underline{x},M))$ as it's $i$-th homology group.

    Let $\underline{x}=x_1, \ldots ,x_n $ be a sequence in $A$ and $M$ is an $A$-module. Comparing the $i$-th homology groups $H_i(K\textbf{.}(\underline{x},M))$ and $H_i(K\textbf{.}(x',M))(\text{where }x'=x_1, \ldots ,x_{n-1})$, we get the following short exact sequence, for all $i\geq 0$ $$0\to H_0(x_n,H_i(K\textbf{.}(x',M)))\to H_i(K\textbf{.}(\underline{x},M))\to H_1(x_n,H_{i-1}(K\textbf{.}(x',M)))\to 0.$$
\end{point}
\begin{point}
    (Sally Descent)\normalfont
$\,$ Let $(A,\m)$ be local ring, $M$ an $A$-module of dimension $d$. Let $x_1,\ldots ,x_r$ is an $M$-superficial sequence with $r< d$. Set $B=A/(x_1,\ldots ,x_r)$ and $N=M/(x_1,\ldots ,x_r)M$.
 Then $$\operatorname{depth }_{G(B)}G(N)\geq 1 \operatorname{ iff } \operatorname{depth }_{G(A)}G(M)\geq r+1.$$
\end{point}

\section{Length of $i$-th Koszul homology of associated graded modules}

In this section $\R$ denotes the Rees ring $A[\m t]$.
\begin{theorem}
\label{uv}
     Let $(A,\m)$ be a Cohen Macaulay local ring with an infinite residue field. $M$ be a finitely generated Cohen Macaulay $A$ module of dimension $r\geq 1$. Let $Xt=x_1t,\ldots ,x_rt \in \mathcal{R}_1.$ Then $\lambda(H_i(Xt,L(M)))< \infty, \text{ for all }i\geq 1$, where $\underline{x}=x_1,\ldots ,x_r$ is an $M$-superficial sequence.
\end{theorem}
\begin{proof}
    We will prove this by induction on $r$. Set $L=L(M)$. Let $r=1$. We have an exact sequence
    \begin{equation}{\tag{$\star$}}
        0\to \mathcal{B}(x_1,M)\to L(M)(-1)\xrightarrow{\phi} L(M)\to L(N)\to 0,
    \end{equation}
 where $\phi(a+\m^nM)=x_1a+\m^{n+1}M$, $\mathcal{B}(x_1,M)=\bigoplus_{n\geq 0}(\m^{n+1}M:x_1/\m^{n}M)$ and $N=M/x_1M$. From above exact sequence, we get $H_1(x_1t, L(M))= \bigoplus_{n\geq 0}(\m^{n+1}M:x_1/\m^{n}M).$ Since $x_1$ is $M$-superficial regular $(\m^{n+1}M:x_1)/\m^{n}M=0$ for $n\gg 0.$ Hence \\ $\lambda(H_1(x_1t,L(M)))<\infty.$ Now assume it is true for $r=s$. we have $$0\to H_0(x_{s+1}t,H_i(X't,L))\to H_i(Xt,L)\to H_1(x_{s+1}t,H_{i-1}(X't,L)\to 0,$$ where $X'=x_1,\ldots ,x_s$. By induction $\lambda(H_i(Xt,L(M)))<\infty $ for all $i\geq 2$. Since $H_0(X't,L(M))=L(M/X'M)$ and as $x_{s+1}$ is $M/X'M$-superficial, by $(\star)$ we get \\  $H_1(x_{s+1}, H_0(X't,L(M))) $ has finite length. Hence $\lambda(H_i(Xt,L(M)))<\infty $ for all $i\geq 1.$
\end{proof}
Next we prove
\begin{lemma}
\label{Thm:lemma}
     Let $(A,\m)$ be a regular local ring with an infinite residue field. $M$ be a Cohen Macaulay $A$-module of dimension $r$. Let $x_1,\ldots ,x_r$ be a $M\oplus A$-superficial sequence. Then $H_i(Xt,Yt,L(M)) \text{ for all }i> s$ has finite length, where $X=x_1,\ldots ,x_r$ is $M$-superficial sequence and  $\m=<x_1,\ldots ,x_r,y_1,\ldots ,y_s>$(minimal set of generators).
\end{lemma}
\begin{proof}
    Set $L=L(M)$. We will prove that $H_i(Xt,y_1t,\ldots ,y_lt ,L)$ has finite length for $i>l$  by induction on $l$. Let $l=1$. We  have the following short exact sequence $$0\to H_0(y_1t,H_i(Xt,L))\to H_i(Xt,y_1t,L)\to H_1(y_1t,H_{i-1}(Xt,L))\to 0.$$ By Theorem \ref{uv}, we get for all $i\geq 2> 1$, $ H_i(Xt,y_1t,L)$ has finite length. Now assume it is true for $l=c.$ Similarly as above we have the following short exact sequence $$0\to H_0(y_{c+1}t,H_i(Xt,Y't,L))\to H_i(Xt,Yt,L)\to H_1(y_{c+1}t,H_{i-1}(Xt,Y't,L)) \to 0,$$ where $Y'=y_1,\ldots ,y_c.$ Since for $i>c$, $H_i(Xt,y_1t,\ldots ,y_ct,L)$ has finite length. So for $i>c+1$, $H_i(Xt,y_1t,\ldots ,y_{c+1}t,L)$ has finite length.
\end{proof}
The main result of this section is:
\begin{theorem}
\label{Thm:main}
    Let $(A,\m)$ be a regular local ring of dimension $d$ and let $M$ be a  Cohen Macaulay $A$-module. Set $R=G(A)=k[X_1,\ldots ,X_d]$. Set $p=\operatorname{projdim}(G(M))$. If $G(M)$ is not Cohen Macaulay then $\alpha_p(M)<\alpha_{p-1}(M).$
\end{theorem}
\begin{proof}
  We may assume the residue field of $A$ is infinite. Let the dimension of $M$ is $r$ and $x_1,\ldots ,x_r$ be a $M\oplus A$-superficial sequence with respect to $\m.$  Since $A$ is regular with infinite residue field, so $x_1,\ldots ,x_r$ is part of regular system of parameters of $A$. Let $\m=\langle x_1,\ldots ,x_r,y_1\ldots ,y_l \rangle$ be minimal set of generator. Here $1\leq r\leq d-1$, so $l\geq 1.$ Let depth $G(M)=c < r(0\leq c\leq r).$ If $c\geq 1$ then $x^*_1,\ldots ,x^*_c$ is $G(M)$ regular(see \cite{HCCMM}, Theorem 8) and $$G(M)/(x_1^*,\ldots ,x_c^*)G(M)\cong G(M/(x_1,\ldots ,x_c)M).$$  Set $N=M/(x_1,\ldots ,x_c)M$ and $S=G(A/(x_1,\ldots ,x_c))=R/(x_1^*,\ldots ,x_c^*).$ We also have $\Tor^R_i(k,G(M))\cong \Tor^S_i(k,G(N))$ for all $i$(see \cite{matsumura}, p-140). So we can now assume depth $G(M)=0.$ dim$M=\text{ dim }G(M)=r\geq 1$ and dim$A=d=r+l\geq r+1.$

  We have a short exact sequence $$0\to G(M)\to L(M)\to L(M)(-1)\to 0.$$ This induces the following exact sequence $$0\to H_d(Xt,Yt, G(M))\to H_d(Xt,Yt,L(M))\to H_d(Xt,Yt,L(M)(-1))\to \ldots.$$ Set $V_i=H_i(Xt,Yt, G(M))$ and $W=H_d(Xt,Yt,L(M))$. From above exact sequence we get $$0\to V_{d,n}\to W_n\to W_{n-1}\to V_{d-1,n}\to \ldots .$$ We know $V_{d-1,n}=0$ for all $n\geq \alpha_{p-1}(M)+1.$ So for all $n\geq \alpha_{p-1}(M)+1$ we get surjections $$W_{n+3}\twoheadrightarrow W_{n+2}\twoheadrightarrow W_{n+1}\twoheadrightarrow W_{n-1}\to 0.$$ Since $r\geq 1$ and $d=r+l>l$. By Lemma \ref{Thm:lemma}, the length of $W=H_d(Xt, Yt, L(M))$ is finite. So $W_{n+j}=0 \text{ for all }j\gg 0.$ Thus $W_j=0 \text{ for }j\geq \alpha_{p-1}.$ So $V_{d,j}=0\text{ for all }j\geq \alpha_{p-1}.$ Hence $\alpha_p(M)<\alpha_{p-1}(M).$

\end{proof}
As a consequence we get:
\begin{corollary}
    (With hypothesis as in \ref{Thm:main}) If $G(M)$ has a quasi-pure resolution then $G(M)$ is Cohen Macaulay.
\end{corollary}
\begin{proof}
    Suppose $G(M)$ is not Cohen Macaulay. Set $p=\operatorname{ projdim}(G(M))$. By above theorem $\alpha_p<\alpha_{p-1}.$ But $G(M)$ has quasi pure resolution, so $\alpha_p\geq \gamma_p\geq \alpha_{p-1}.$ This is a contradiction.
\end{proof}
We give some examples which show our assumptions are optimal. We used Singular \cite{Singular} to check the following examples.
\begin{example}
\label{def:example}
 \normalfont
\begin{enumerate}
    \item  Let $A=k[[x,y]]$, $I=(x^2,xy)$ and $M=A/I$. It is clear that $M$ is not Cohen Macaulay and is of dimension 1. Note  $G(M)=R/(X^2,XY)$, where $R=k[X,Y]$. Now we have the following pure  graded resolution of $G(M)$ $$0\to R(-3)\to R^2(-2)\to R\to 0.$$ Here $\alpha_2=3$ and $\alpha_1=2$, $i.e.$ $\alpha_2 >\alpha_1.$
    \item Let $A=k[[x,y]]$ and $M=k.$ Then $G(M)=k$ is Cohen Macaulay. From the Koszul complex of $G(M)$ we get $\alpha_i=i \text{ for all } i,\, 1\leq i\leq d.$ This implies $\alpha_d>\alpha_{d-1}.$ So Theorem \ref{Thm:main} is false if $G(M)$ is Cohen Macaulay.

    \item  Let $A=k[[x,y]]$, $I=(x^3,x^2y,y^4)$ and $M=A/I$. So $R=G(A)=k[X,Y]$ and $G(M)=R/(X^3,X^2Y,Y^4).$ Clearly $M$ and $G(M)$ is Cohen Macaulay because they are of dimension zero. Graded resolution of $G(M)$ is $$0\to R(-4)\oplus R(-6)\to R^2(-3)\oplus R(-4)\to R\to 0$$ and it is quasi pure resolution. So there exist Cohen Macaulay $A$-modules with quasi-pure resolution.
    \item  Let $A=k[[x,y]]$, $I=(x^2,xy,y^5)$ and $M=A/I$. So $R=G(A)=k[X,Y]$ and $G(M)=R/(X^2,XY,Y^5).$ Clearly $M$ and $G(M)$ is Cohen Macaulay because they are  of dimension zero. Graded resolution of $G(M)$ is $$0\to R(-3)\oplus R(-6)\to R^2(-2)\oplus R(-5)\to R\to 0$$ and it is not a quasi pure resolution. This shows there exists $M$ Cohen Macaulay with $G(M)$ not having a quasi-pure resolution.

\end{enumerate}
\end{example}
\section{Hilbert coefficient of the modules of finite projective dimension}
 In [\cite{Loweylength}, Theorem 1.1] it first shows that if $(A,\m)$ is a non-regular Gorenstein local ring and $G(A)$ is Cohen Macaulay,  then for all finite length $A$-modules $E$ of finite projective dimension $\ell \ell(E)\geq \operatorname{reg}(G(A))+1$ (here $\ell\ell(E)$ denotes the Lowey length of $E$.  Set $c=\operatorname{reg}(G(A))$.   In this section, we first prove if $M$ is Cohen Macaulay of positive dimension with finite projective dimension then $e_0(M)\geq \mu(M)+ reg(G(A))$ and $e_1(M)\geq \binom{c+1}{2}$. Then we will show $G(M)$ is Cohen Macaulay if equality holds for $e_1(M)$(see \ref{Thm:reg}).
\begin{point}
\normalfont
Recall the Lowey length of a finite length $A$-module $E$ is defined to be the number $$\ell \ell(E)=\operatorname{min}\{i\, |\, \m^iE=0\}.$$
Let $G(A)_+$ be the irrelevant maximal ideal of $G(A)$. Let $H^i(G(A))$ be the $i$-th local cohomology module of $G(A)$ with respect to $G(A)_+$. The Castelnuovo-Mumford regularity of $G(A)$ is $$\operatorname{reg}(G(A))=\operatorname{max}\{i+j\, |\,H^i(G(A))_j\neq 0 \}.$$ One can easily check that if $x^*$ is $G(A)$-regular element of degree $1$ then $\operatorname{reg}(G(A))=\operatorname{reg}(G(A/xA)).$
\end{point}

\begin{lemma}
\label{Thm:lowerbound}
Let $(A,\m)$ be Gorenstein local ring. Assume $G(A)$ is Cohen Macaulay. Let $M$ be Cohen Macaulay module of dimension $r$ with finite projective dimension. Then $e_0(M)\geq \mu(M)+ c$ and $e_1(M)\geq \binom{c+1}{2},$ where $c=\operatorname{reg}(G(A)).$
\end{lemma}
\begin{proof}
We may assume that the residue field $k$ is infinite. Let  $x_1,\ldots ,x_r$ be $A\oplus M$-superficial sequence. Set $B=A/(x_1,\ldots ,x_r)$ and $N=M/(x_1,\ldots ,x_r)M(\text{i.e. } \lambda(N)<\infty).$   Let  $\ell \ell(N)=t,$ i.e. $G(N)=\bigoplus_{i=0}^{t}\m^iN/\m^{i+1}N.$  Set $a_i=\lambda(\m^iN/\m^{i+1}N).$ The Hilbert series of $N$ is $H_N(z)=\mu(N)+a_1z+\ldots +a_{t-1}z^{t-1}$. So
\begin{align*}
    e_0(N)=\mu(N)+a_1+\ldots +a_{t-1}\geq \mu(N)+t-1\geq \mu(N)+\operatorname{reg}(G(A)).
\end{align*}
 We also have
 \begin{align*}
     e_1(N)= & a_1+2a_2+\ldots +(t-1)a_{t-1} \\ \geq & 1+2+\ldots +(t-1)\geq 1+2+\ldots +c \\ = & \binom{c+1}{2}.
 \end{align*}
  By (\cite{HCCMM}, Corollary 10) $\mu(M)=\mu(N)$, $e_0(M)=e_0(N)$ and $e_1(M)\geq e_1(N).$ Hence $e_0(M)\geq \mu(M)+\operatorname{reg}(G(A))$ and $e_1(M)\geq \binom{c+1}{2}.$
\end{proof}

The next theorem is about what happens if $e_1(M)=\binom{c+1}{2}.$
\begin{theorem}
\label{Thm:reg}
   Let $(A,\m)$ be a Gorenstein local ring with $G(A)$ be Cohen Macaulay. Let $M$ be Cohen Macaulay module of dimension $r$ having finite projective dimension. If $e_1(M)=\binom{c+1}{2}$, then $G(M)$ is Cohen Macaulay.
\end{theorem}
\begin{proof}
     Let $x_1,\ldots ,x_r$ be $A\oplus M$-superficial sequence. Set $N=M/(x_1,\ldots ,x_{r-1})M.$ By (\cite{HCCMM}, Corollary 10) $e_1(M)=e_1(N)$ and $e_1(N)\geq e_1(N/x_rN)$. So by the hypothesis $\binom{c+1}{2}=e_1(N)\geq e_1(N/x_rN)\geq \binom{c+1}{2}.$ So $e_1(N)=e_1(N/x_rN)$. By (\cite{HCCMM}, Corollary 10) $G(N)$ is Cohen Macaulay. So by Sally descent $G(M)$ is Cohen Macaulay.
\end{proof}
\section{Cohen Macaulay modules over complete intersection rings}
In this section, we will prove an interesting result on Cohen Macaulay modules over complete intersection ring(see Theorem \ref{Projective}) and we will also give an application of the Corollary \ref{Thm:lowerbound} and Theorem \ref{Thm:reg}. For this, we need to recall the following things.
\begin{point}
   \normalfont
   Let $A$ be a local ring and $M$ be an $A$-module. Let $\beta_i^A(M)=\lambda(\operatorname{Tor}_i(M,k))$ be the $i$-th Betti number of $M$ over $A.$ The complexity of $M$ over $A$ is defined as $$\textrm{cx}_A(M)=\inf \left\{b\in \N \quad \vline \quad {\limsup_{n\to \infty}} \,\, \frac{\beta_n^A(M)}{n^{b-1}}< \infty \right\}.$$
   Note that $M$ has bounded Betti numbers iff $\textrm{cx}_A(M)\leq 1.$
\end{point}

\begin{point}
\label{Thm:5.2}
\normalfont
Let $Q$ be a local ring and $\textbf{f}=f_1,\ldots ,f_c$ be a $Q$-regular sequence. Set $A=Q/(\textbf{f}).$ The Eisenbud operators(see \cite{Eisenbud}) are constructed as follows

Let $\F:\cdots F_{i+2}\xrightarrow{\partial}F_{i+1}\xrightarrow{\partial}F_i\ldots$ be a complex of free $A$-modules.\\
(a) Choose a sequence of free $Q$-modules $\wt{\F_i}$ and maps $\wt{\partial}$ between them $$\wt{\F}:\cdots \wt{F}_{i+2}\xrightarrow{\wt{\partial}}\wt{F}_{i+1}\xrightarrow{\wt{\partial}}\wt{F}_i\ldots$$ such that $\F=\wt{\F}\otimes A$.\\
(b) Since $\wt{\partial}^2\equiv 0$ modulo$(f_1,\ldots ,f_c),$ we may write  $\wt{\partial}^2=\sum_{j=1}^{c} f_j\wt{t_j}$ where $\wt{t_j}:\wt{F_i}\to \wt{F}_{i-2}$ are linear maps for all $i.$\\
(c) Define, for $j=1,\ldots ,c$ the map $t_j=t_j(Q,\textbf{f},\F)$ by $t_j=\wt{t_j}\otimes A.$

The operators $t_1,\ldots ,t_c$ are called Eisenbud operators associated to $\textbf{f}.$ It can be shown that
\begin{enumerate}
    \item $t_i$ are uniquely determined up to homotopy.
    \item $t_i,t_j$ commutes up to homotopy.
\end{enumerate}
\end{point}
\begin{point}
\label{Thm:basechange}
    \normalfont
    Let $\textbf{g}=g_1,\ldots ,g_c$ be another $Q$-regular sequence such that $(\textbf{g})=(\textbf{f})$. The Eisenbud operators associated to $\textbf{g}$ can be constructed by using  Eisenbud operators associated to $\textbf{f}$ as follows:

    Let $f_i=\sum_{j=1}^{c}a_{ij}g_j \text{ for }i=1,\ldots ,c$. Then we can choose $t_i'=\sum_{j=1}^{c}a_{ji}t_c \text{ for }i=1,\ldots ,c$ as Eisenbud operators for $g_1,\ldots ,g_c.$ Note that in \cite{Eisenbud} there is an indexing error.
\end{point}
\begin{point}
\label{Thm:5.3}
    \normalfont
    Let $R=A[t_1,\ldots ,t_c]$ be a polynomial ring over $A$ with variable $t_1,\ldots ,t_c$ of degree $2.$ Let $M$ be an $A$-module  and let $\F$ be a free resolution of $M$. So we get well defined  maps $$t_j:\Ext_A^n(M,k)\to \Ext_A^{n+2}(M,k) \text{ for all }1\leq j\leq c \text{ and all }n.$$  This turns
    $\operatorname{Ext}_A^*(M,k)=\bigoplus_{i\geq 0}\Ext_A^i(M,k)$ into a $R$-module. Where $k$ is residue field of $A.$

    Since $\m\subseteq \text{ann}(\Ext_A^i(M,k))$ for all $i\geq 0$ we get that $\operatorname{Ext}_A^*(M,k)$ is a $S=R/\m R=k[t_1,\ldots ,t_c]$-module.
\end{point}

\begin{remark}
\label{Thm:identified}
\normalfont
    The subring $k[t_1,\ldots ,t_r]$$(r<c)$ of $S$ can be identified with the ring $S'$ of cohomological operators of a presentation $A=P/(f_1,\ldots ,f_r)$, where $P=Q/(f_{r+1},\ldots ,f_c).$
\end{remark}

\begin{point}
\label{Thm:5.4}
\normalfont
    In (\cite{Gulliksen}, 3.1) Gulliksen proved that if $\text{projdim}_Q(M)$ is finite then $\operatorname{Ext}_A^*(M,N)$ is finitely generated $R$-module for all  $A$-modules $N$. In (\cite{Avramov}, 3.10) Avramov proved a converse for $N=k.$ Note that if $\text{projdim}_Q(M)$ is finite then \, $\operatorname{ Ext}_A^*(M,k)$ is finitely generated graded $S$-module of Krull dimension $\textrm{cx}_A(M).$
\end{point}
Let $(A,\m)$ be a local ring. For $x\in \m^i\setminus \m^{i+1}$ we set $\text{ord}(x)=i$ as order of $x.$
\begin{theorem}\label{Projective}

 Let  $(Q,\n)$ be a regular local ring with infinite residue field, $f_1,\ldots ,f_c\in \n^2$ of order $s$ such that $f_1^*,\ldots ,f_c^* $ is a $G(Q)$-regular sequence.  Let $A=Q/(f_1,\ldots ,f_c)$ and $M$ be Cohen Macaulay $A$-module with $\operatorname{cx}_A(M)=r< c$. Then $M$ has finite projective dimension as $Q/(g_{r+1},\ldots ,g_c)$-module, for some  $g_{r+1},\ldots ,g_c$. Here  $g_{r+1}^*,\ldots ,g_c^*$ is a $G(Q)$-regular sequence and $\operatorname{ord}(g_i)=s.$

\end{theorem}
\begin{proof}
      By \ref{Thm:5.4}, $\operatorname{Ext}_A^*(M,k)$ is finitely generated graded $S$-module of Krull dimension $\textrm{cx}_A(M).$ Set $E=\operatorname{Ext}_A^*(M,k)$. Let $\xi_1,\ldots , \xi_r\in S_2$ be a system of parameters of $E$ and $\xi_i=\sum_{j=1}^{c}\overline{\beta_{ij}}t_j.$ Set $\beta=(\overline{\beta_{ij}})_{r\times c}$. Since $\xi_1,\ldots , \xi_r\in S_2$ is a system of parameters, the rank of matrix $\beta$ is equals to $r.$

     Note that the first $r$ columns of $\beta$ may not be linearly independent. But we can rearrange the columns of $\beta$ so that the first $r$ columns are linearly independent.  Without loss of generality, we can reorder $f_1,\ldots ,f_c$ so that the first $r$ column of $\beta$ is linearly independent(see \ref{Thm:basechange}). Now set      $\alpha=(a_{ij})_{c\times c}$, where $$a_{ij}=  \left\{
\begin{array}{ll}
      \beta_{ij} & 1\leq i\leq r, \, 1\leq j\leq c \\
      1 & \text{if } i=j \text{ and } i>r \\
     0 & \text{otherwise} \\

\end{array}
\right. $$
(here we assume the image of $\beta_{ij}$ in $k$ is $\overline{\beta_{ij}}$.)
Clearly, the matrix $\alpha$ is nonsingular(as $\overline{\alpha}$ is non-singular). Set $\xi_j=t_j$ for $j>r.$ Then $S=k[\xi_1,\ldots , \xi_c].$ By (\cite{HFMCMpart2}, 1.9) there is a regular sequence $\textbf{g}=g_1,\ldots ,g_c$ such that $(\textbf{g})=(\textbf{f})$ and if $t_i'$ are the Eisenbud operators associated to $\textbf{g}$, then the action of $t_j'$ on $E$ is same as that of $\xi_j$ for $j=1,\ldots ,c$.

By (\cite{HFMCMpart2}, 1.9), $[\textbf{g}]=(\alpha^{tr})^{-1}[\textbf{f}]$, where $\alpha^{tr}$ is the transpose of $\alpha$.
We note that $\ord(g_i) \geq s$ for all $i$. As $(\textbf{g})=(\textbf{f})$ we get $A = Q/(g_1,\ldots, g_c)$. We have $e(A) = s^c$. So by \cite[1.8]{RV} it follows that $\ord(g_i) = s$ for all $i$ and $g_1^*, \ldots, g_c^*$ is a $G(Q)$-regular sequence. We also have $G(A) = G(Q)/(g_1^*,\ldots, g_c^*)$.

The subring $k[\xi_1,\ldots ,\xi_r]$ of $S$ can be identified by the ring $S'$ of cohomological operators of a presentation $A=P/(g_1,\ldots ,g_r)$, where $P=Q/(g_{r+1},\ldots ,g_c)$. Since $\operatorname{Ext}_A^*(M,k)$ is finitely generated graded $S'$-module, we obtained by \ref{Thm:5.4} that $\text{projdim}_P(M)$ is finite. This proves our result.
\end{proof}

As a consequence, we obtain:
\begin{theorem}
\label{Thm:invariant}
Let  $(Q,\n)$ be a regular local ring with infinite residue field, $f_1,\ldots ,f_c\in \n^2$ of order $s$ such that $f_1^*,\ldots ,f_c^* $ is a $G(Q)$-regular sequence.  Let $A=Q/(f_1,\ldots ,f_c)$ and $M$ be Cohen Macaulay $A$-module with $\operatorname{cx}_A(M)=r< c$.
Then $e_0(M)\geq \mu(M)+\alpha$ and $e_1(M)\geq \binom{\alpha+1}{2}$,  where $\alpha=(c-r)(s-1).$ Moreover if $e_1(M)=\binom{\alpha+1}{2}$ then $G(M)$ is Cohen Macaulay.
\end{theorem}
\begin{proof}
    We may assume that $A$ has an infinite residue field(see \cite{HFMCMpart2}, 1.1).
    By Theorem \ref{Projective}, we get that $A=B/(g_1,\ldots ,g_r)$ and $M$ has finite projective dimension as $B$-module,  where $B=Q/(g_{r+1},\ldots ,g_c)$ and $(g_1^*,\ldots ,g_c^*)$ is a $G(Q)$-regular sequence. Since  $\text{ord}(g_i)=s$ for all $1\leq i\leq c$, we obtain that $\operatorname{reg}(G(B))=(c-r)(s-1)$. Set $\alpha=(c-r)(s-1).$ By Lemma \ref{Thm:lowerbound},  $e_0(M)\geq \mu(M)+\alpha$ and $e_1(M)\geq \binom{\alpha+1}{2}.$  By Theorem \ref{Thm:reg} if $e_1(M) = \binom{\alpha+1}{2}$, then $G(M)$ is Cohen Macaulay $G(B)$-module (and so a Cohen-Macaulay $G(A)$-module).
\end{proof}


\begin{thebibliography}{10}

\bibitem{ananthnarayan2023associated}
H.~Ananthnarayan, Manav Batavia, and Omkar Javadekar, \emph{When do associated
  graded modules have pure resolutions?}, arXiv preprint arXiv:2308.00654
  (2023).

\bibitem{Avramov}
Luchezar~L. Avramov, \emph{Modules of finite virtual projective dimension},
  Inventiones mathematicae \textbf{96} (1989), no.~1, 71--101.

\bibitem{Loweylength}
Luchezar~L. Avramov, Ragnar-Olaf Buchweitz, Srikanth~B. Iyengar, and Claudia
  Miller, \emph{Homology of perfect complexes}, Advances in Mathematics
  \textbf{223} (2010), no.~5, 1731--1781.

\bibitem{Eisenbud}
David Eisenbud, \emph{Homological algebra on a complete intersection, with an
  application to group representations}, Transactions of the American
  Mathematical Society \textbf{260} (1980), no.~1, 35--64.

\bibitem{Singular}
Gert-Martin Greuel, Gerhard Pfister, and Hans Sch{\"o}nemann,
  \emph{Singular—a computer algebra system for polynomial computations},
  Symbolic computation and automated reasoning, AK Peters/CRC Press, 2001,
  pp.~227--233.

\bibitem{Gulliksen}
Tor~H. Gulliksen, \emph{A change of ring theorem with applications to
  poincar{\'e} series and intersection multiplicity}, Mathematica Scandinavica
  \textbf{34} (1974), no.~2, 167--183.

\bibitem{matsumura}
Hideyuki Matsumura, \emph{Commutative ring theory}, no.~8, Cambridge university
  press, 1989.

\bibitem{HCCMM}
Tony~J. Puthenpurakal, \emph{Hilbert coefficients of a {Cohen}--{Macaulay}
  module}, Journal of Algebra \textbf{264} (2003), no.~1, 82--97.

\bibitem{Part1}
\bysame, \emph{Ratliff--{Rush} filtration, regularity and depth of higher
  associated graded modules: Part {I}}, Journal of Pure and Applied Algebra
  \textbf{208} (2007), no.~1, 159--176.

\bibitem{HFMCMpart2}
\bysame, \emph{The {Hilbert} function of a maximal {Cohen}--{Macaulay} module.
  {Part II}}, Journal of Pure and Applied Algebra \textbf{218} (2014), no.~12,
  2218--2225.

\bibitem{Pure}
\bysame, \emph{On associated graded modules having a pure resolution},
  Proceedings of the American Mathematical Society \textbf{144} (2016), no.~10,
  4107--4114.

\bibitem{RV}
M.~E. Rossi and G.~Valla, \emph{Multiplicity and {$t$}-isomultiple ideals},
  Nagoya Math. J. \textbf{110} (1988), 81--111.

\end{thebibliography}
\providecommand{\bysame}{\leavevmode\hbox to3em{\hrulefill}\thinspace}
\providecommand{\MR}{\relax\ifhmode\unskip\space\fi MR }
\providecommand{\MRhref}[2]{%
  \href{http://www.ams.org/mathscinet-getitem?mr=#1}{#2}
}
\providecommand{\href}[2]{#2}




\end{document}